\documentclass{amsart}
\usepackage{parskip}
\usepackage{graphicx}
\usepackage{import}
\usepackage{xifthen}
\usepackage{pdfpages}
\usepackage{transparent}
\usepackage[export]{adjustbox}
\usepackage{svg}
\usepackage{quiver}
\usepackage{tikz-cd}
\usepackage{amsthm}
\usepackage{amsmath}
\usepackage{amsfonts}
\usepackage{amssymb}
\usepackage[normalem]{ulem}
\usepackage[utf8x]{inputenc}
\usepackage{lmodern}
\usepackage{microtype}
\usepackage{hyperref}

\newcommand{\R}{\mathbf{R}}

\newcommand{\Z}{\mathbf{Z}}

\newcommand{\Ker}{\textrm{Ker}}

\newcommand{\Hom}{\textrm{Hom}}
\newcommand{\Foam}{\textrm{Foam}}
\newcommand{\slt}{\mathfrak{sl}_2}

\newtheorem{theorem}{Theorem}[section]
\newtheorem{proposition}{Proposition}[section]
\newtheorem{corollary}{Corollary}[section]
\newtheorem{lemma}{Lemma}[section]
\newtheorem{remark}{Remark}[section]
\theoremstyle{definition}
\newtheorem{definition}{Definition}[section]
\allowdisplaybreaks

\begin{document}

\title{An \(\mathfrak{sl}_2\) action on link homology of \(T(2,k)\) torus links}
\author{Felix Roz}
\maketitle

\begin{abstract}
We determine an \(\mathfrak{sl}_2\)-module structure on the equivariant Khovanov--Rozansky homology of \((2,k)\)-torus links following the framework in \cite{qrsw-homology}.
\end{abstract}

\section{Introduction}
In 1984, Vaughan Jones discovered the Jones polynomial, an integral Laurent polynomial-valued invariant of oriented links that admits a combinatorial definition via the Kauffman bracket \cite{Jones_1985, Kauffman_1987}.
In 2000, Mikhail Khovanov defined a categorification of the Jones polynomial: a bigraded homology theory whose Euler characteristic is the Jones polynomial \cite{khovanov-homology}.
Khovanov's original construction involved a simple combinatorial definition which was shown to be functorial with respect to link cobordisms by several authors
\cite{cmw, sano, blanchet, caprau, beliakova}.
Further work by Khovanov and Rozansky gave rise to a homology theory which categorifies the \(\mathcal{U}_q(\mathfrak{sl}_N)\)-polynomial  using matrix factorizations \cite{kr-matrix}.
The work of Robert and Wagner provided an explicit evaluation formula which can be used to construct Khovanov--Rozansky homology with webs and foams in a combinatorial manner similar to Khovanov's original construction \cite{rw-evaluation}.

In order to understand the structure of the homology theory, Khovanov and Rozansky found an action of the positive half of the Witt algebra on HOMFLYPT homology \cite{kr-witt}.
Following this work, Qi, Robert, Sussan, and Wagner found that the Witt algebra acts on foams
and a copy of \(\mathfrak{sl}_2\) contained in the Witt algebra acts on link homology \cite{qrsw-foams, qrsw-homology}.
In this paper, we determine the structure of \(H^*(T_{2,k})\), the homology of a \((2,k)\)-torus link, as a representation of \(\mathfrak{sl}_2\).
Sections \ref{webs-foams}, \ref{action}, and \ref{link-homology} review webs, foams, the action of \(\mathfrak{sl}_2\), and link homology following \cite{qrsw-homology}.
Section \ref{homology-computation} contains the computation of the complex \(C^*(T_{2,k})\) in the \(\mathfrak{sl}_2\)-equivariant setting.
In section \ref{sl2-computation}, the homology \(H^*(T_{2,k})\) is decomposed as an \(\slt\)-representation into indecomposables.

\subsection*{Acknowledgments}
I am especially thankful to Joshua Sussan for introducing me to categorification, link homology, and foams and for guiding me through the work behind this paper. I am also thankful to Mikhail Khovanov, Alexis Gu\'{e}rin, Matt Hogancamp, Taketo Sano, and Josh Wang for helpful discussions and comments on earlier versions of this paper.

\section*{Conventions}

In this paper, \(\mathbf{k}\) is a field.
For any \(x \in \mathbf{k}\), define \(\overline{x} := 1- x\).
For any positive integer \(n\), let \(R_n := \mathbf{k}[x_1, \cdots, x_n]^{S_n}\) where \(\deg(x_i) =2\).
In this ring, \(e_i, p_i,\) and \(h_i\) denote the \(i\)th elementary, power sum, and complete homogeneous polynomials.

Let \(N\) be a fixed positive integer.
Elements of the ring \(R_N\) will be distinguished by capital letters.
For example, \(X_i, E_i, P_i, H_i\) refer to the indeterminates and the elementary, power sum, and complete homogeneous symmetric polynomials in \(R_N\).

\section{Webs and Foams}
In this section, we review \(\mathfrak{gl}_N\)-webs and foams. For more details see \cite{rw-evaluation}.
\label{webs-foams}
\begin{definition}
A \(\mathfrak{gl}_N\)-\textit{web} is a finite, oriented, trivalent graph \(\Gamma = (V, E)\) embedded in \(\R^2\),
together with a \textit{thickness function} \(\ell: E \to \Z_{> 0}\)
that satisfies a flow condition: no vertex may be a source or a sink and the sum of the thicknesses of the incoming edges must equal the sum of the thicknesses of the outgoing edges. The flow condition ensures that every vertex in a web is one of three local models:
\[
\includegraphics[valign=c]{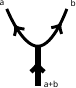}
\quad \textrm{or} \quad
\includegraphics[valign=c]{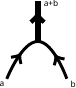}
\]
The first is called a \textit{split}, the second a \textit{merge}.
In these local pictures, the edge with thickness \(a+b\) is called the \textit{thick} edge relative to the vertex, while the edges with thicknesses \(a\) and \(b\) are called the \textit{thin} edges relative to the vertex.
\end{definition}

\begin{definition}
A \(\mathfrak{gl}_N\)-\textit{foam} in \(\R^2 \times [0,1]\) is a finite collection \(F\) of compact, oriented surfaces called facets, glued together along their boundaries
together with a \textit{thickness function} \(\ell: F \to \Z_{>0}\)
such that every point has a closed neighborhood homeomorphic to one of the following:
\begin{enumerate}
\item A disk,
\item A cylinder over a merge or a split, denoted \(Y^{(a,b)}\):
\[\includegraphics[height=3cm]{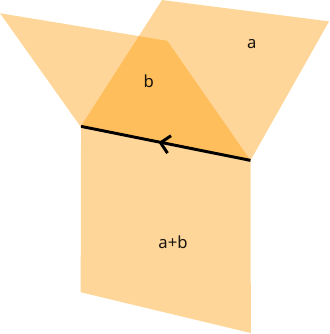},\]
\item A cone over the 1-skeleton of a tetrahedron, denoted \(T^{(a,b,c)}\):
\[\includegraphics[height=3cm]{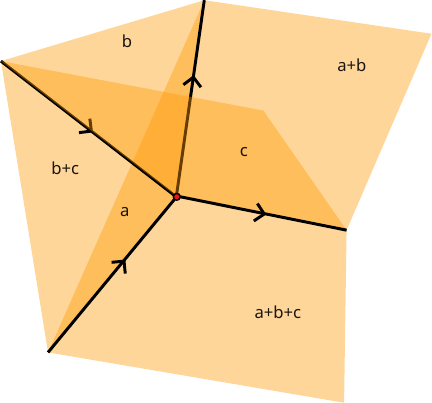}.\]
\end{enumerate}

The set of points of the second type is a collection of curves called the \textit{bindings}.
The points of the third type are called the \textit{singular vertices}.
For each foam \(F\), denote the set of facets by \(F^2\), the set of bindings by \(F^1\), and the set of singular vertices by \(F^0\).
In the pictures above, the bindings are marked by black lines and the singular vertices are marked by red dots.

In the \(Y^{(a,b)}\) case, the thicknesses of the facets agree with the thicknesses of the edges in the merge or split.
The binding is oriented so that it agrees with the thin facets and is opposite to the orientation of the thick facet.

The \textit{boundary} \(\partial F\) of \(F\) is the closure of the set of boundary points of facets that are not contained in any bindings.
A foam with empty boundary is called a \textit{closed foam}.

Finally, each facet \(f \in F\) may be given a \textit{decoration} \(p_f \in R_{\ell(f)} \otimes R_{N - \ell(f)}\).
If a facet is decorated by the identity, the decoration is omitted in diagrams.
On a facet with thickness \(l\),
the symbol \(\spadesuit_i\) denotes the decoration \(p_i \otimes 1 \in R_{l} \otimes R_{N-l}\).
The symbol
\(\hat{\spadesuit}_i\) denotes the decoration \(1 \otimes p_i \in R_{l} \otimes R_{N-l}\).
\end{definition}

\begin{definition}
For \(f \in F^2\), define
\[\deg_N(f) := \ell(f)(N - \ell(f))\chi(f).\]
Let \(F^1_{\--}\) denote the collection of bindings diffeomorphic to intervals and \(F^1_{\circ}\) the bindings diffeomorphic to circles.
For \(s \in F^1_{\--}\) with a neighborhood diffeomorphic to \(Y^{(a,b)}\), define
\[\deg_N(s) := ab + (a+b)(N-a-b).\]
For \(v \in F^0\) with a neighborhood diffeomorphic to \(T^{(a,b,c)}\), define
\[\deg_N(v) := ab + bc +ac + (a+b+c)(N - a - b -c).\]
Finally, for any decorated foam \(F\) with decorations \(\{p_f\}_{f \in F^2}\) define
\[
\deg_N(F) :=
\sum_{f \in F^2} \deg(P_f) -
\sum_{f \in F^2} \deg_N(f) +
\sum_{s \in F^1_{\--}} \deg_N(s) -
\sum_{v \in F^0} \deg_N(v).
\]
\end{definition}

\subsection{Basic Foams}
The following foams are called \textit{basic foams}. Every foam is isotopic to a foam which is composed of basic foams glued along their boundaries.

Polynomial:
\[
\includegraphics[height=2cm,valign=c]{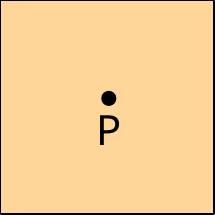}
\]

Associativity:
\[
\includegraphics[height=2cm,valign=c]{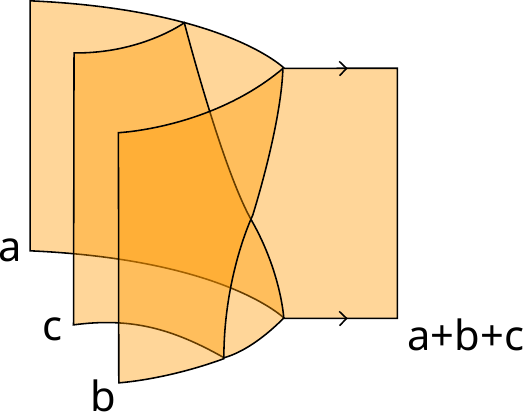}
\quad\quad\quad\quad
\includegraphics[height=2cm,valign=c]{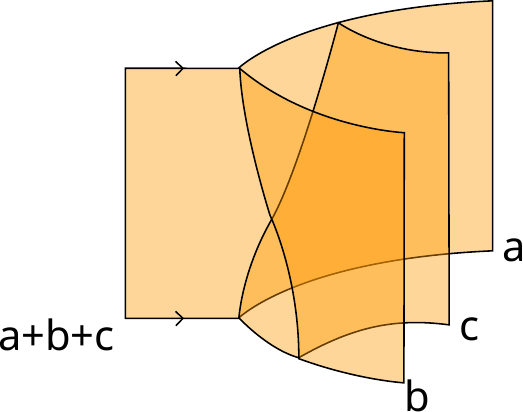}
\]

Digon cup and cap:
\[
\includegraphics[height=2cm,valign=c]{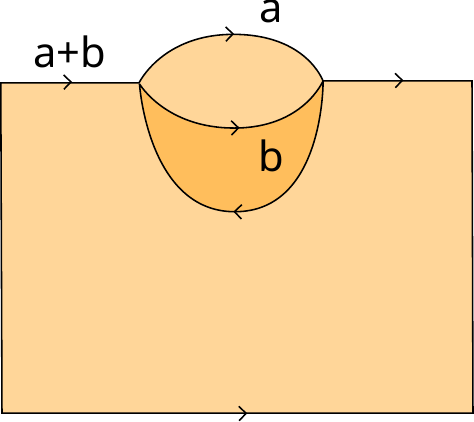}
\quad\quad\quad\quad
\includegraphics[height=2cm,valign=c]{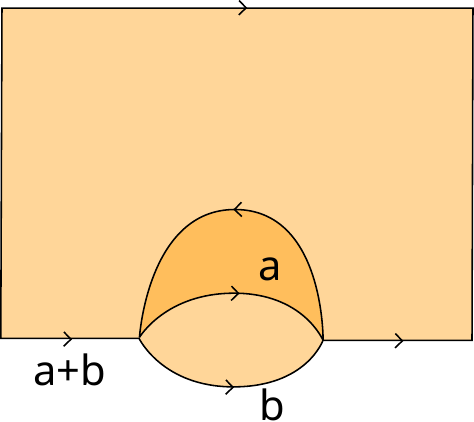}
\]

Zip and unzip:
\[
\includegraphics[height=2cm,valign=c]{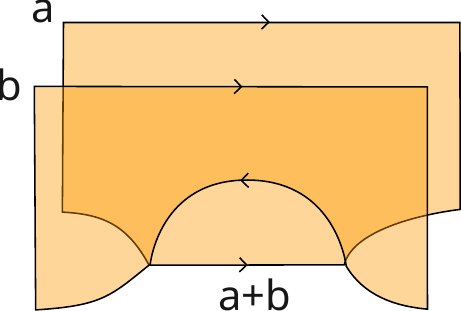}
\quad\quad\quad\quad
\includegraphics[height=2cm,valign=c]{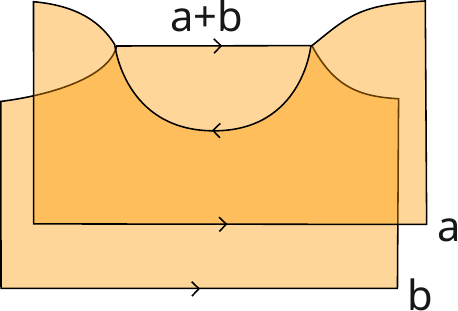}
\]

Cup and cap:
\[
\includegraphics[height=2cm,valign=c]{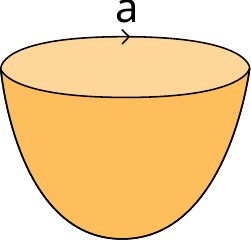}
\quad\quad\quad\quad
\includegraphics[height=2cm,valign=c]{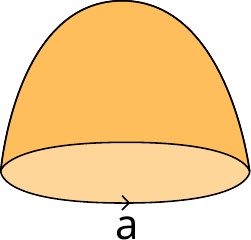}
\]

Saddle:
\[
\includegraphics[height=2cm,valign=c]{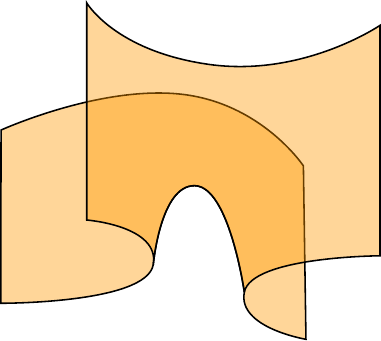}
\]

\subsection{The category of webs and foams}
\label{category-foam}
Suppose \(F\) is a foam embedded in \(\R^2 \times [0,1]\)
such that 
\(F_0 := F \cap (\R^2 \times \{0\})\)
and 
\(F_1 := F \cap (\R^2 \times \{1\})\)
are both webs.
Then \(F\) is called a foam between the webs \(F_0\) and \(F_1\).
By convention, foams are read from bottom to top in diagrams.
If \(G\) is another web with \(G_0 = F_1\), then \(G\) and \(F\) can be glued together along the common web to form a composition \(G \circ F\).
Finally, for any web \(W\) there is a foam \(W \times I\) called the identity foam for \(W\).

\begin{definition}
\label{category-def}
Let \(\widehat{\mathrm{Foam}}_N\) be the category whose objects are isotopy classes of webs and and whose morphisms are foams between webs with composition given by stacking and identities given by identity foams.
The category \(\Foam_N\) is formed from \(\widehat{\mathrm{Foam}}_N\) by adding formal degree shifts of objects and taking the \(R_N\)-linear additive closure.
\end{definition}

\subsection{Foam evaluation}
\label{foam-eval}
In \cite{rw-evaluation}, Robert and Wagner defined
an \textit{evaluation} \(\langle F \rangle \in R_N\)
for every closed \(\mathfrak{gl}_N\)-foam \(F\).
For any web \(\Gamma\), define an \(R_N\)-bilinear form \(\langle \cdot ; \cdot \rangle_N\) on \(\Hom_{\mathbf{Foam}}(\varnothing, \Gamma)\) by
\begin{equation}
\langle F ; G \rangle_N := \langle \overline{G} \circ F \rangle_N,
\end{equation}
where \(\overline{G}: \Gamma \to \varnothing\) is the reflection of a foam \(G: \varnothing \to \Gamma\) along the plane \(\R^2 \times \{1/2\}\).

\begin{definition}
\label{def-state-space}
For any web \(\Gamma\), 
\begin{equation}
\mathcal{F}_N(\Gamma) := \Hom(\varnothing, \Gamma) / \mathrm{Rad}\langle \cdot ; \cdot \rangle_N
\end{equation}
is called the \textit{state space} of \(\Gamma\).
By the universal construction of \cite{Blanchet_1995} the state space extends to a functor \(\mathcal{F}_N: \Foam_N \to R_N\text{-mod}\).
\end{definition}

\section{Action of \(\mathfrak{sl}_2\) on Foams}
\label{action}

In this section, we recall the action of the \(\slt\) on foams originally defined in \cite{qrsw-foams}.

Let \(\mathfrak{sl}_2\) be the Lie algebra over \(\mathbf{k}\) generated by symbols \(\mathbf{e}, \mathbf{h}, \mathbf{f}\) with the relations
\[
[\mathbf{h},\mathbf{e}] = 2\mathbf{e}, \quad [\mathbf{h}, \mathbf{f}]= -2 \mathbf{f}, \quad [\mathbf{e}, \mathbf{f}] = \mathbf{h}.
\]

The following action of \(\mathfrak{sl}_2\) on foams was defined in \cite{qrsw-homology}.

\begin{definition}[Action on polynomials]
\label{action-poly}
For every integer \(k \geq -1\), the linear differential operator \(L_k\) is defined on \(p \in R_n\) by
\begin{equation}
L_k \cdot p = - \sum_{i=1}^n x_i^{k+1} \frac{\partial}{\partial x_i}.
\end{equation}
The map \(\slt \to \mathrm{End}(R_n)\) given by
\begin{equation}
\begin{aligned}
\mathbf{e} &\mapsto L_{-1}\\
\mathbf{h} &\mapsto 2L_0 \\
\mathbf{f} &\mapsto -L_{1}
\end{aligned}
\end{equation}
defines an action of \(\mathfrak{sl}_2\) on \(R_n\).
\end{definition}

\begin{remark}
Since \(\deg(x_i) = 2\), the operator \(\mathbf{h}\) multiplies any polynomial by its degree. Therefore the \(\lambda\)-weight space of \(R_n\) as an \(\mathfrak{sl}_2\)-representation is the subspace of homogeneous polynomials with degree \(\lambda\).
\end{remark}

\begin{definition}[Action on basic foams]
\label{action-basic}
Let \(t_1, t_2 \in \mathbf{k}\) be fixed parameters.
The Lie algebra \(\mathfrak{sl}_2\) acts trivially on traces of isotopies and according to Definition \ref{action-poly} on polynomials.
The generator \(\mathbf{e}\) acts trivially on all other foams and the generators \(\mathbf{h}\) and \(\mathbf{f}\) act
according to the following rules on all other basic foams:


\[
\mathbf{h} \cdot 
\includegraphics[height=2cm,valign=c]{assoc-left.pdf}
=
\mathbf{h} \cdot
\includegraphics[height=2cm,valign=c]{assoc-right.pdf}
=
0
\]

\[
\mathbf{h} \cdot
\includegraphics[height=2cm,valign=c]{digon-cup.pdf}
=
ab(t_1+t_2)\;
\includegraphics[height=2cm,valign=c]{digon-cup.pdf}
\]

\[
\mathbf{h} \cdot
\includegraphics[height=2cm,valign=c]{digon-cap.pdf}
=
ab(\overline{t_1}+\overline{t_2})\;
\includegraphics[height=2cm,valign=c]{digon-cap.pdf}
\]

\[
\mathbf{h} \cdot
\includegraphics[height=2cm,valign=c]{basic-unzip.pdf}
=
-ab(\overline{t_1}+\overline{t_2})\;
\includegraphics[height=2cm,valign=c]{basic-unzip.pdf}
\]

\[
\mathbf{h} \cdot
\includegraphics[height=2cm,valign=c]{zip.pdf}
=
-ab({t_1}+{t_2})\;\;
\includegraphics[height=2cm,valign=c]{zip.pdf}
\]

\[
\mathbf{h} \cdot
\includegraphics[height=2cm,valign=c]{basic-cup.pdf}
=
a(N - a)\;
\includegraphics[height=2cm,valign=c]{basic-cup.pdf}
\]

\[
\mathbf{h} \cdot
\includegraphics[height=2cm,valign=c]{cap.pdf}
=
a(N - a)\;
\includegraphics[height=2cm,valign=c]{cap.pdf}
\]

\[
\mathbf{h} \cdot
\includegraphics[height=2cm,valign=c]{saddle.pdf}
=
-a(N - a)\;
\includegraphics[height=2cm,valign=c]{saddle.pdf}
\]


\[
\mathbf{f} \cdot 
\includegraphics[height=2cm,valign=c]{assoc-left.pdf}
=
\mathbf{f} \cdot
\includegraphics[height=2cm,valign=c]{assoc-right.pdf}
=
0
\]

\[
\mathbf{f} \cdot
\includegraphics[height=2cm,valign=c]{digon-cup.pdf}
=
-t_1\;
\includegraphics[height=2cm,valign=c]{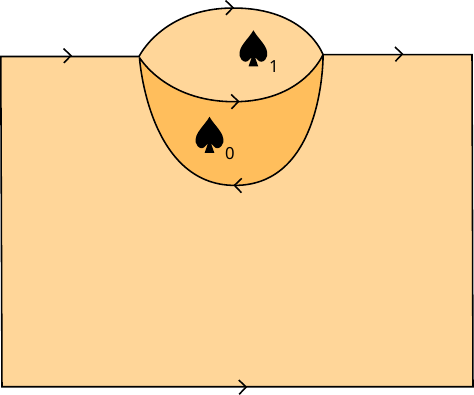}
-t_2\;
\includegraphics[height=2cm,valign=c]{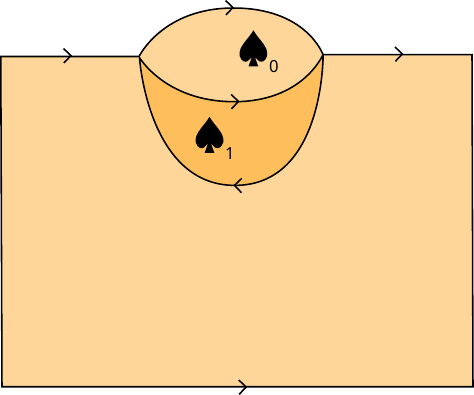}
\]

\[
\mathbf{f} \cdot
\includegraphics[height=2cm,valign=c]{digon-cap.pdf}
=
-\overline{t_1}\;
\includegraphics[height=2cm,valign=c]{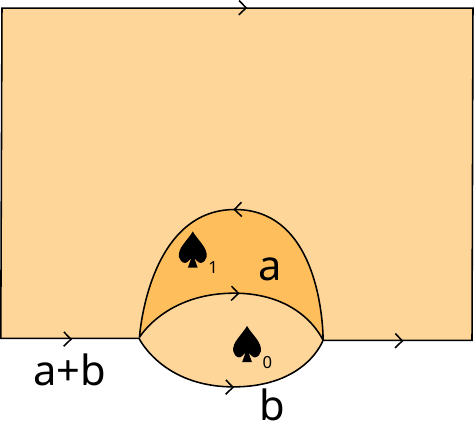}
-\overline{t_2}\;
\includegraphics[height=2cm,valign=c]{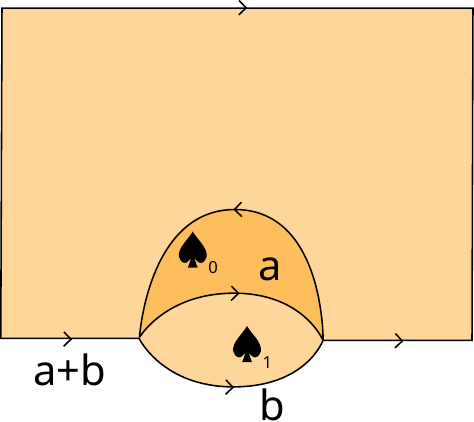}
\]

\[
\mathbf{f} \cdot
\includegraphics[height=2cm,valign=c]{basic-unzip.pdf}
=
\overline{t_1}\;
\includegraphics[height=2cm,valign=c]{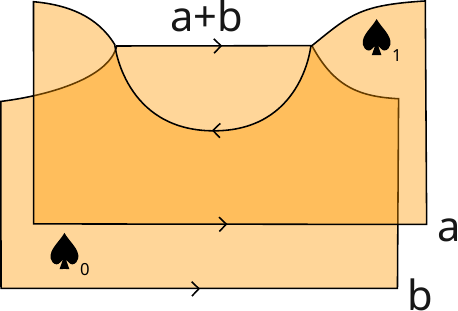}
+
\overline{t_2}\;
\includegraphics[height=2cm,valign=c]{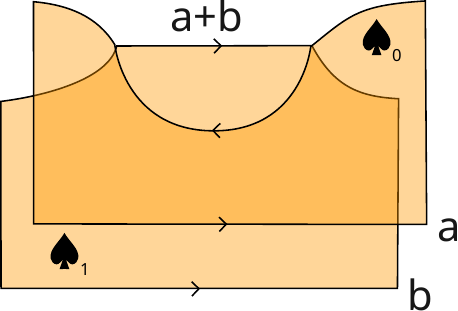}
\]

\[
\mathbf{f} \cdot
\includegraphics[height=2cm,valign=c]{zip.pdf}
=
t_1\;
\includegraphics[height=2cm,valign=c]{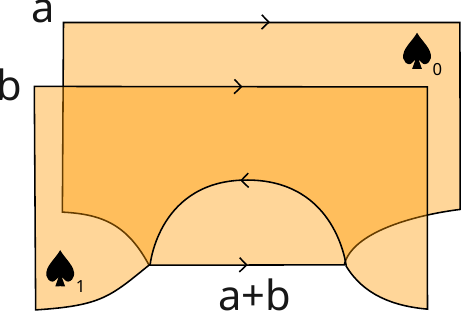}+
t_2\;
\includegraphics[height=2cm,valign=c]{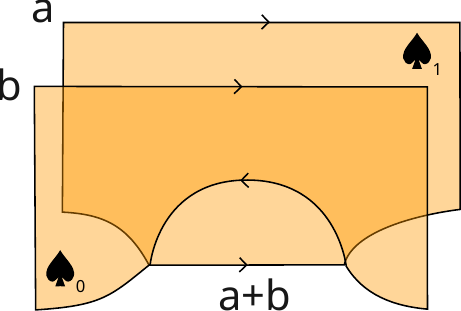}
\]

\[
\mathbf{f} \cdot
\includegraphics[height=2cm,valign=c]{basic-cup.pdf}
=
-\frac{1}{2}\;
\includegraphics[height=2cm,valign=c]{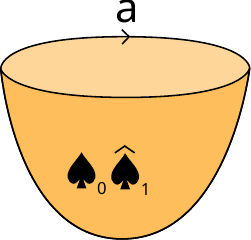}
-\frac{1}{2}\;
\includegraphics[height=2cm,valign=c]{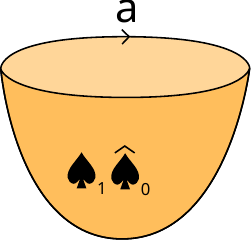}
\]

\[
\mathbf{f} \cdot
\includegraphics[height=2cm,valign=c]{cap.pdf}
=
-\frac{1}{2}\;
\includegraphics[height=2cm,valign=c]{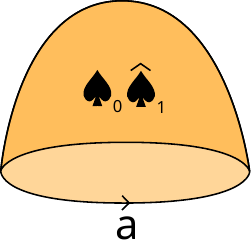}
-\frac{1}{2}\;
\includegraphics[height=2cm,valign=c]{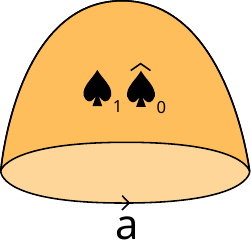}
\]

\[
\mathbf{f} \cdot
\includegraphics[height=2cm,valign=c]{saddle.pdf}
=
\frac{1}{2}\;
\includegraphics[height=2cm,valign=c]{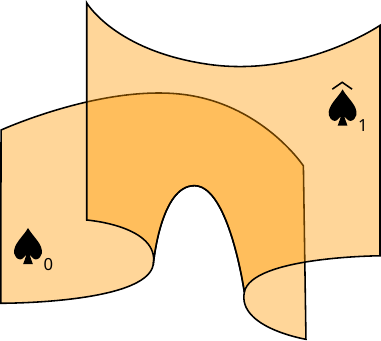}+
\frac{1}{2}\;
\includegraphics[height=2cm,valign=c]{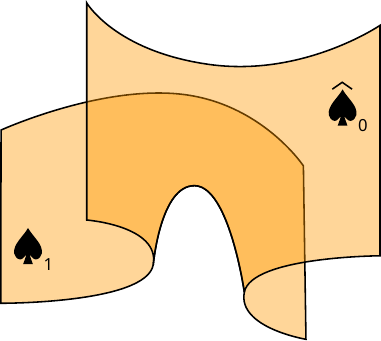}
\]
\end{definition}

\begin{definition}[Action on any foam]
\label{action-all}
The action of \(\mathfrak{sl}_2\) is extended to arbitrary foams so that it acts by derivations with respect to gluings of foams.
More precisely:
\[g \cdot (G \circ H) = (g \cdot G) \circ H + G \circ (g \cdot H)\]
where \(g \in \mathfrak{sl}_2\), \(G\) and \(H\) are foams, and \(\circ\) denotes a gluing of foams along a common boundary component.
Since every foam is isotopic to a collection of basic foams glued along components of their boundaries, this defines an action on all foams.
\end{definition}


\begin{proposition}[{\cite[Lemma 3.9, Proposition 3.10]{qrsw-homology}}]
This action of \(\mathfrak{sl}_2\) on foams is well-defined, invariant under isotopy, and for any closed foam \(F\)
\[
\langle g \cdot F \rangle = g \cdot \langle F \rangle.
\]
The action on the right side of the equation is the action on \(R_N\) from Definition \ref{action-poly}.
\end{proposition}

\subsection{Twisted Actions}
\label{twists}

In Section \ref{link-homology}, chain complexes of webs and foams will define the \(\mathfrak{gl}_N\)-homology of links.
In order for the homology of the complex to be a representation of \(\mathfrak{sl}_2\), we require the differential to induce \(\mathfrak{sl}_2\)-equivariant maps between state spaces.
Unfortunately, the standard differentials used to define \(\mathfrak{gl}_N\) do not induce equivariant maps.
In this section, we define the machinery used to adjust the action on state spaces to make the differentials equivariant.

\begin{definition}
\label{sl2-category}
An \textit{\(\mathfrak{sl}_2\)-category} \(\mathcal{C}\) is a category such that for any pair of objects \(A,B\), the set of morphisms, \(\Hom(A,B)\), is a representation of \(\mathfrak{sl}_2\) and for any
\(g \in \mathfrak{sl}_2\), \(G \in \Hom(A,B)\), and \(H \in \Hom(C, A)\)
\[g \cdot (G\circ H) = (g \cdot G) \circ H + G \circ (g \cdot H).\]
\end{definition}

\begin{corollary}
The \(\mathfrak{sl}_2\) action on foams defines an \(\mathfrak{sl}_2\)-category structure on the category \(\mathrm{Foam}_N\).
\end{corollary}

\begin{lemma}
\label{eq-condition}
For any pair of webs \(A, B\), a foam \(F \in \mathrm{Hom}(A,B)\) induces an \(\mathfrak{sl}_2\)-equivariant map on state spaces if and only if \(\mathfrak{sl}_2\) acts trivially on \(F\).
\end{lemma}
\begin{proof}
The foam \(F\) induces an equivariant map if 
\[
g \cdot Fv = (g \cdot F)v + F(g \cdot v) = F(g \cdot v)
\]
for all \(v \in \Hom(\varnothing, A)\) and \(g \in \mathfrak{sl}_2\).
This only happens if \(g \cdot F = 0\).
\end{proof}

Since \(\mathfrak{sl}_2\) acts non-trivially on basic foams, this condition is hard to come by naturally.
To remedy this, the action on state spaces can be ``twisted''.

\begin{definition}
\label{def-twist}
A \textit{twist} on an object \(A\) in an \(\mathfrak{sl}_2\) category \(\mathcal{C}\) is a \(\mathbf{k}\)-linear map \(\tau_{(-)}: \mathfrak{sl}_2 \to \mathrm{End}(A)\) such that for all \(g, h \in \mathfrak{sl}_2\)
\begin{equation}
\tau_{[g,h]} = [\tau_g, \tau_h] + g \cdot \tau_h - h \cdot \tau_g,
\end{equation}
where \([\tau_g,\tau_h]\) denotes the commutator of endomorphisms.
\end{definition}

\begin{definition}
\label{def-twisted-category}
For any \(\mathfrak{sl}_2\)-category \(\mathcal{C}\), there is an \(\slt\)-category \(Tw(\mathcal{C})\) of \textit{twisted} \(\mathcal{C}\) \textit{objects}.
Its objects are pairs \((A,\tau)\) of objects in \(\mathcal{C}\) and twists of those objects.
The sets of morphisms have the same underlying \(\mathbf{k}\)-modules
\[
\Hom_{Tw(\mathcal{C})}((A,\tau), (B, \sigma)) := \Hom_{\mathcal{C}}(A,B),
\]
with a new \(\mathfrak{sl}_2\) action defined by
\begin{equation}
\label{signs}
g * f := g \cdot f + \sigma_{g} \circ f - f \circ \tau_g,
\end{equation}
where \(g \cdot f\) is the action in \(\mathcal{C}\).
\end{definition}

\subsection{Green Dots}
\label{green-dots}

Suppose \(e\) is an edge in a web \(A\). For any decoration \(p \in R_{\ell(e)} \otimes R_{N - \ell(e)}\), there is an endomorphism \(p^{(e)}\) of \(A\) defined by the identify foam for \(A\) with the decoration \(p\) on the facet attached to edge \(e\). 

\begin{definition}
\label{green-twist}
The \textit{hollow twist} on the edge \(e\) is defined on generators of \(\mathfrak{sl}_2\) by
\begin{equation}
\begin{aligned}
\mathbf{e} &\mapsto 0\\
\mathbf{h} &\mapsto -\spadesuit^{(e)}_0 \\
\mathbf{f} &\mapsto \spadesuit^{(e)}_1.
\end{aligned}
\end{equation}
The \textit{solid twist} on the edge \(e\) is defined on generators of \(\mathfrak{sl}_2\) by
\begin{equation}
\begin{aligned}
\mathbf{e} &\mapsto 0\\
\mathbf{h} &\mapsto -\hat{\spadesuit}^{(e)}_0 \\
\mathbf{f} &\mapsto \hat{\spadesuit}^{(e)}_1.
\end{aligned}
\end{equation}
\end{definition}

These twist endomorphisms are local to an edge \(e\),
so they can be represented diagrammatically with hollow and solid green dots on the edge \(e\) of a web.
Edges decorated with multiple green dots labelled with scalars are interpreted as linear combinations.

The following are local pictures of the twisted \(\slt\)-action on a foam
between two webs with green dot twists.
Note that new action defined in \eqref{signs} imposes alternating signs depending on if the twist is on the source or target web. In a state space, where all foams map into a given web, only the first column applies.
\[
\mathbf{h} \cdot \includegraphics[height=1.5cm,valign=c]{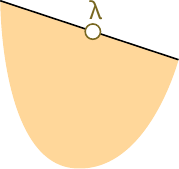}
=
\mathbf{h} \cdot \includegraphics[height=1.5cm,valign=c]{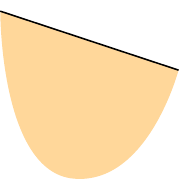}
- \lambda
\includegraphics[height=1.5cm,valign=c]{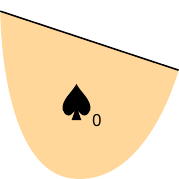},
\quad\quad
\mathbf{h} \cdot \includegraphics[height=1.5cm,valign=c]{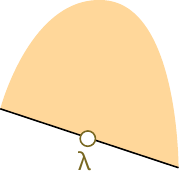}
=
\mathbf{h} \cdot \includegraphics[height=1.5cm,valign=c]{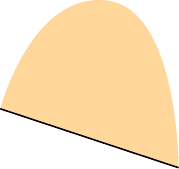}
+ \lambda
\includegraphics[height=1.5cm,valign=c]{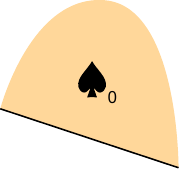}
\]

\[
\mathbf{h} \cdot \includegraphics[height=1.5cm,valign=c]{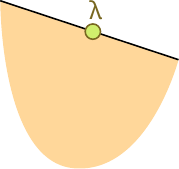}
=
\mathbf{h} \cdot \includegraphics[height=1.5cm,valign=c]{open-top.pdf}
- \lambda
\includegraphics[height=1.5cm,valign=c]{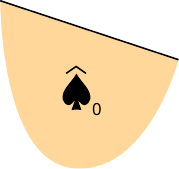},
\quad\quad
\mathbf{h} \cdot \includegraphics[height=1.5cm,valign=c]{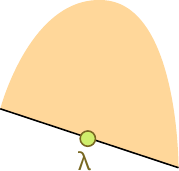}
=
\mathbf{h} \cdot \includegraphics[height=1.5cm,valign=c]{open-bottom.pdf}
+ \lambda
\includegraphics[height=1.5cm,valign=c]{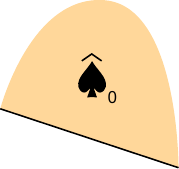}
\]

\[
\mathbf{f} \cdot \includegraphics[height=1.5cm,valign=c]{open-green-top.pdf}
=
\mathbf{f} \cdot \includegraphics[height=1.5cm,valign=c]{open-top.pdf}
+ \lambda
\includegraphics[height=1.5cm,valign=c]{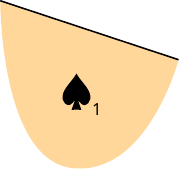},
\quad\quad
\mathbf{f} \cdot \includegraphics[height=1.5cm,valign=c]{open-bottom-green.pdf}
=
\mathbf{f} \cdot \includegraphics[height=1.5cm,valign=c]{open-bottom.pdf}
- \lambda
\includegraphics[height=1.5cm,valign=c]{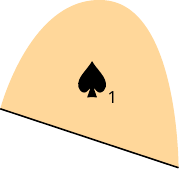}
\]

\[
\mathbf{f} \cdot \includegraphics[height=1.5cm,valign=c]{open-green-top-filled.pdf}
=
\mathbf{f} \cdot \includegraphics[height=1.5cm,valign=c]{open-top.pdf}
+ \lambda
\includegraphics[height=1.5cm,valign=c]{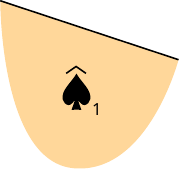},
\quad\quad
\mathbf{f} \cdot \includegraphics[height=1.5cm,valign=c]{open-bottom-green-filled.pdf}
=
\mathbf{f} \cdot \includegraphics[height=1.5cm,valign=c]{open-bottom.pdf}
- \lambda
\includegraphics[height=1.5cm,valign=c]{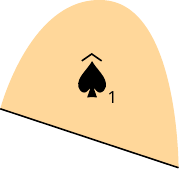}
\]

\section{construction of \(\slt\)-equivariant link homology}
\label{link-homology}

In this section, we recall the construction of an \(\slt\)-equivariant link homology originally defined in \cite{qrsw-homology}.

In diagrams of webs and foams,
facets and edges of thickness 1 are colored blue and facets and edges of thickness 2 are be colored red.
On blue facets, \(n\) dots represent the decoration \(x^n \in R_1 = \mathbf{k}[x]\).

\subsection{Braiding Complexes}
Let the diagrams \(A\) and \(B\) denote the positive and negative crossings respectively:
\[
A = \includegraphics[height=1.5cm, valign=c]{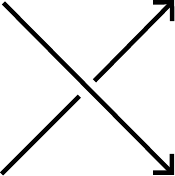},
\quad\quad
B = \includegraphics[height=1.5cm, valign=c]{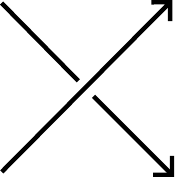}.
\]

Recall that \(t_1,t_2 \in \mathbf{k}\) are two fixed parameters that the action depends on.
Link homology is defined by assigning cohomological braiding complexes to the two crossing types:
\[
C^*(A) =
q\includegraphics[height=1cm,valign=c]{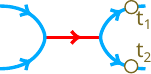}\quad
\xrightarrow{\quad\;\includegraphics[height=1cm,valign=c]{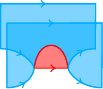}\quad\;}\quad
\uwave{\includegraphics[height=1cm,valign=c]{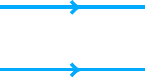}},
\]
\[
C^*(B) =
\uwave{\includegraphics[height=1cm,valign=c]{parallel.pdf}}\quad
\xrightarrow{\quad\;\includegraphics[height=1cm,valign=c]{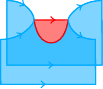}\quad\;}\quad
q^{-1} \includegraphics[height=1cm,valign=c]{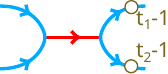}.
\]
In both cases the underlined term is in homological degree 0.
The hollow dot twists ensure that the differentials induce \(\mathfrak{sl}_2\)-equivariant maps.

\begin{definition}
For an arbitrary link \(L\), define \(C^*(L)\) to be total complex of the hypercube of resolutions constructed by taking the tensor product of braiding complexes corresponding to its component crossings.
The homology of \(\mathcal{F}_N(C^*(L))\) is the homology \(H^*(L)\) of the link.
\end{definition}

\subsection{The Relative Homotopy Category}
The action of \(\mathfrak{sl}_2\) on \(R_N\) makes \(R_N\) a \(\mathcal{U}(\mathfrak{sl}_2)\)-module algebra.
Any \(R_N\)-module with an action of \(\mathfrak{sl}_2\) that is compatible with the action on \(R_N\) is a module over the smash product \(\mathcal{U}(\mathfrak{sl}_2) \# R_N\) and any morphism of \(R_N\)-modules which is also \(\mathfrak{sl}_2\)-equivariant is a morphism of \(\mathcal{U}(\mathfrak{sl}_2) \# R_N\)-modules.

Let \(\text{For: }\mathcal{U}(\mathfrak{sl}_2) \# R_N\text{-}\mathrm{Mod} \to R_N\text{-}\mathrm{Mod}\) be the functor which forgets the \(\mathfrak{sl}_2\) action.
This functor extends to a functor on the corresponding categories of chain complexes and their usual homotopy categories.

\begin{definition}
\label{relative-homotopy}
The \textit{relative homotopy category} is the Verdier quotient
\[
\mathcal{K}^{\mathfrak{sl}_2}(R_N) := \frac{\mathcal{K}(\mathcal{U}(\mathfrak{sl}_2) \# R_N)}{\Ker(\mathrm{For})}.
\]
\end{definition}
More explicitly, an \(\slt\)-equivariant chain complex of \(\slt \# R_N\)-modules is relative null homotopic if it is null homotopic as a chain complex of \(R_N\)-modules.
In practice, this means chain homotopy equivalences can be constructed in the relative homotopy category such that one direction is \(\mathfrak{sl}_2\)-equivariant, while the other direction is not.
The weaker, relative condition is enough to guarantee that homology groups are isomorphic as \(\mathfrak{sl}_2\)-representations.

\begin{theorem}[{\cite[Theorem 4.3]{qrsw-homology}}]
As an \(\slt\)-representation, \(H^*(L)\) is an isotopy invariant of framed links in the relative homotopy category.
\end{theorem}

\begin{remark}
As an \(\slt\)-representation, the homology is only a \textit{framed} invariant of links because the Reidemeister I moves come with quantum grading shifts and twists:
\begin{equation}
\includegraphics[valign=c, height=0.75cm]{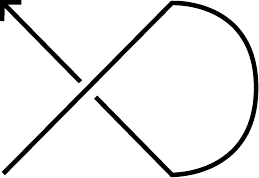}
\cong
q^{-1}
\includegraphics[valign=c, height=0.75cm]{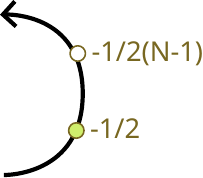},
\quad\quad\quad
\includegraphics[valign=c, height=0.75cm]{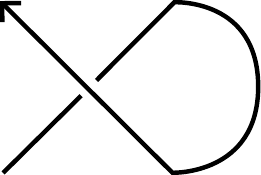}
\cong
q
\includegraphics[valign=c, height=0.75cm]{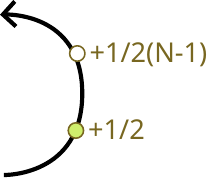}.
\end{equation}
Note that the sign of the twist is opposite to the sign of the crossing.
\end{remark}

\section{The \(\slt\)-equivariant complex \(C^*(T_{2,k})\)}
\label{homology-computation}
In this section, \(N=2\).

\subsection{Local Relations}
Recall a few local relations on foams following \cite{beliakova}:

Dot Reduction by Symmetric Coefficients:
\[
\includegraphics[valign=c, height=0.5cm]{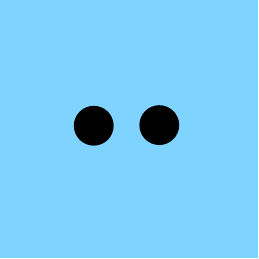}
=
E_1\;
\includegraphics[valign=c, height=0.5cm]{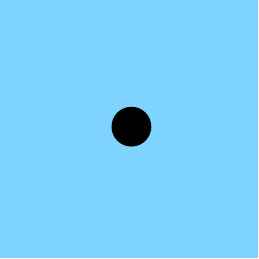}
-
E_2\;
\includegraphics[valign=c, height=0.5cm]{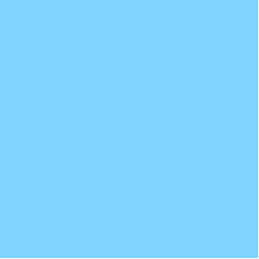}
\]

Dot Migration:
\[
\includegraphics[valign=c, height=1cm]{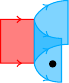}
=
E_1\;
\includegraphics[valign=c, height=1cm]{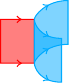}
-
\includegraphics[valign=c, height=1cm]{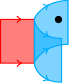}
\]

Neck Cutting:
\[
\includegraphics[valign=c, height=1cm]{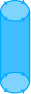}
=
\includegraphics[valign=c, height=1cm]{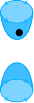}
+
\includegraphics[valign=c, height=1cm]{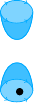}
-
E_1\;
\includegraphics[valign=c, height=1cm]{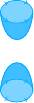},
\quad\quad\quad
\includegraphics[valign=c, height=1cm]{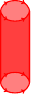} =
-\includegraphics[valign=c, height=1cm]{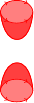}
\]

Spheres:
\[
\includegraphics[valign=c,height=0.5cm]{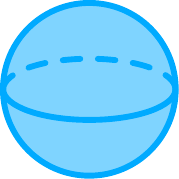} = 0
\quad
\includegraphics[valign=c,height=0.5cm]{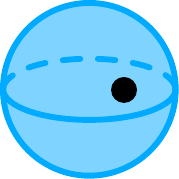} = 1
\quad
\includegraphics[valign=c,height=0.5cm]{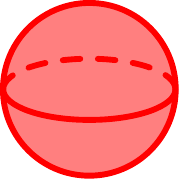} = -1
\]

Detachments:
\[
\includegraphics[valign=c,height=1cm]{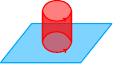} =
-\includegraphics[valign=c,height=1cm]{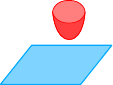},
\quad\quad\quad
\includegraphics[valign=c,height=1cm]{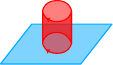} =
\includegraphics[valign=c,height=1cm]{red-neck-detached.pdf}
\]
\[
\includegraphics[valign=c,height=1cm]{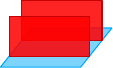} =
-\includegraphics[valign=c,height=1cm]{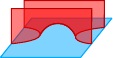},
\quad\quad\quad
\includegraphics[valign=c,height=1cm]{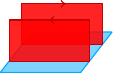} =
\includegraphics[valign=c,height=1cm]{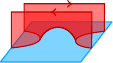}
\]

\subsection{Special Webs and Foams}
The following are distinguished green-dotted webs:
\begin{gather*}
H^n := \includegraphics[valign=c]{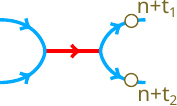},\\
\Phi^n := \includegraphics[valign=c]{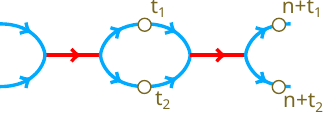},\\
I := \includegraphics[valign=c]{parallel.pdf}.
\end{gather*}

The following are distinguished foams between the above webs:

\[\epsilon := \includegraphics[valign=c]{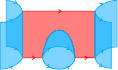} \quad
\epsilon^{\bullet}_C := \includegraphics[valign=c]{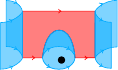} \quad
\epsilon^{\bullet}_R := \includegraphics[valign=c]{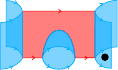}\quad
\iota := \includegraphics[valign=c]{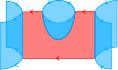}\]

\[m_L := \includegraphics[valign=c]{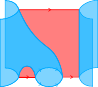}\quad
m_R := \includegraphics[valign=c]{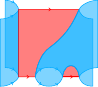}\]

\[
h := \includegraphics[valign=c]{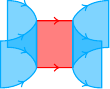}\quad
h^{\bullet}_L := \includegraphics[valign=c]{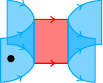}\quad
h^{\bullet}_R := \includegraphics[valign=c]{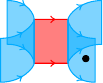}\]

\[\phi^{\bullet}_R := \includegraphics[valign=c]{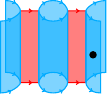}\quad
\phi^{\bullet}_C := \includegraphics[valign=c]{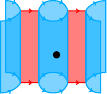}\]

\[z := \includegraphics[valign=c]{zip-bottom.pdf}\quad
u := \includegraphics[valign=c]{unzip.pdf}.\]

\subsection{The complex \(C^*(A_k)\)}
The homology of \(T_{2,k}\) is determined by computing the complex of a link which contains a sequence of positive crossings.
Let \(A_k\) denote the sequence of \(k\) positive crossings.

\begin{proposition}
\label{twist-complex}
In the relative homotopy category, the complex \(C^*(A_k)\) is isomorphic to
\begin{equation}
	0 \xrightarrow{} q^{2k-1} H^{k-1} \xrightarrow{d^{-k}} q^{2k-3} {H^{k-2}} \xrightarrow{d^{-k+1}} \cdots \xrightarrow{d^{-3}} q^{3} {H^{1}} \xrightarrow{d^{-2}} q {H^{0}} \xrightarrow{d^{-1}} \uwave{I} \xrightarrow{} 0
\end{equation}
where \(d^{-1} := z\), \(d^{-j} := h^{\bullet}_R - h^{\bullet}_L\) when \(j\) is even and \(d^{-j} := h^{\bullet}_R + h^{\bullet}_L - E_1 h\) when \(j\neq 1\) is odd.
The underlined term is in homological degree 0.
\end{proposition}

The proof follows from two lemmas.

\begin{lemma}
\label{basecase}
The following diagram is commutative, has exact columns, and all morphisms are \(\mathfrak{sl}_2\)-equivariant:
\begin{equation}
\label{basecase-diagram}
\begin{tikzcd}[ampersand replacement=\&]
	{q^3H^{1}} \&\& q^{1}{H^{0}} \&\& I \\
	\\
	q^2{\Phi^{0}} \&\& {qH^{0}\oplus qH^{0}} \&\& I \\
	\\
	q{H^{0}} \&\& q{H^{0}} \&\& 0
	\arrow["h", from=5-1, to=5-3]
	\arrow[from=5-3, to=5-5]
	\arrow["{\begin{bmatrix}h\\h\end{bmatrix}}", from=5-3, to=3-3]
	\arrow["\iota"', from=5-1, to=3-1]
	\arrow["\epsilon"', from=3-1, to=1-1]
	\arrow["{\begin{bmatrix}m_R\\m_L\end{bmatrix}}", from=3-1, to=3-3]
	\arrow["{\begin{bmatrix}h&-h\end{bmatrix}}"', from=3-3, to=1-3]
	\arrow["{\begin{bmatrix}z&-z\end{bmatrix}}", from=3-3, to=3-5]
	\arrow[from=5-5, to=3-5]
	\arrow["{h^{\bullet}_R - h^{\bullet}_L}", from=1-1, to=1-3]
	\arrow["z", from=1-3, to=1-5]
	\arrow["id", from=3-5, to=1-5]
\end{tikzcd}.
\end{equation}
\end{lemma}

\begin{proof}
The foams \(m_R, m_L\) and \(z\) are all included in braiding complexes and are thus known to be equivariant.
The foam \(h\) is simply the identity foam and is trivially equivariant.
To show that \(\iota\), \(\epsilon\), and \(h_R^{\bullet} - h_L^{\bullet}\) are equivariant,
it is sufficient to show that the generators \(\mathbf{h}\) and \(\mathbf{f}\) act trivially on these foams.
The computation for the action of \(\mathbf{f}\) on \(\iota\) and \(\epsilon\) is shown below.
The other computations are similar.
Green dots which are common to the source and target are omitted to simplify diagrams.


\begin{align*}
\mathbf{f} \cdot \includegraphics[valign=c]{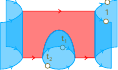} = 
-\overline{t_1} \; &\includegraphics[valign=c]{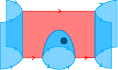}
-\overline{t_2} \; \includegraphics[valign=c]{counit-center-x.pdf} \\
           -t_1 \; &\includegraphics[valign=c]{counit-back-center-x.pdf}
           -t_2 \; \includegraphics[valign=c]{counit-center-x.pdf}
               +\; \includegraphics[valign=c]{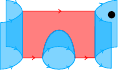}
               +\; \includegraphics[valign=c]{counit-right-x.pdf}\\
= - \; &\includegraphics[valign=c]{counit-back-center-x.pdf}
- \; \includegraphics[valign=c]{counit-center-x.pdf} \\
               +\; &\includegraphics[valign=c]{counit-back-right-x.pdf}
               +\; \includegraphics[valign=c]{counit-right-x.pdf}\\
= \; &\includegraphics[valign=c]{counit-center-x.pdf}
- E_1\; \includegraphics[valign=c]{counit.pdf}
- \; \includegraphics[valign=c]{counit-center-x.pdf} \\
-\; &\includegraphics[valign=c]{counit-right-x.pdf}
+ E_1\; \includegraphics[valign=c]{counit.pdf}
               +\; \includegraphics[valign=c]{counit-right-x.pdf}\\
	       =0.
\end{align*}

\begin{align*}
\mathbf{f} \cdot \left(
\includegraphics[valign=c]{right-x.pdf} -
\includegraphics[valign=c]{left-x.pdf}
\right) = \;
E_1\; &\includegraphics[valign=c]{right-x.pdf} -
E_1\; \includegraphics[valign=c]{left-x.pdf}\\
-&\includegraphics[valign=c]{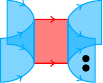}
-\includegraphics[valign=c]{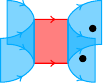}
+\; \includegraphics[valign=c]{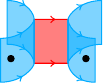}
+\includegraphics[valign=c]{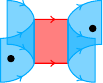}\\
=
E_1\; &\includegraphics[valign=c]{right-x.pdf}
- E_1\; \includegraphics[valign=c]{left-x.pdf}\\
+E_1\; &\includegraphics[valign=c]{left-x.pdf}
-E_1\;\includegraphics[valign=c]{right-x.pdf}\\
=0.
\end{align*}

The bottom left square commutes since \(m_R \circ \iota\), \(m_L \circ \iota\) and \(h\) are isotopic:
\[
m_R \circ \iota = \includegraphics[valign=c]{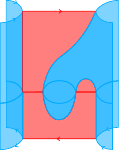}.
\]
The top right square commutes since \(h\) is the identity foam for \(H\).
The top square commutes because of the neck cutting relation:
\[
\includegraphics[valign=c]{right-m.pdf} = 
\includegraphics[valign=c]{counit-center-x.pdf}
+\includegraphics[valign=c]{counit-right-x.pdf}
-E_1 \includegraphics[valign=c]{counit.pdf}.
\]

The composition \(\epsilon \circ \iota\) forms a blue sphere surrounded by a red facet, which can be separated into a free floating blue sphere using the detachment relations.
Since the undotted blue sphere is equal to \(0\), the kernel of \(\epsilon\) consists of the elements of the state space of \(\Phi^0\) which have no dots on the inner circle.
This is exactly image of \(\iota\), so the left column is exact.
Since the \(h\) foams are identities for \(H\), the middle column is trivially exact.
\end{proof}

\begin{corollary}
The middle row of \eqref{basecase-diagram} is \(C^*(D_2)\) and the bottom row is acyclic.
Therefore \(C^*(D_2)\) is isomorphic to the top row in the relative homotopy category.
\end{corollary}

\begin{lemma}
\label{finalsquare}
The following diagram is \(\mathfrak{sl}_2\)-equivariant, commutative, and has exact columns:
\[\begin{tikzcd}[ampersand replacement=\&]
	q^{2k+1}{H^{k}} \&\& q^{2k-1} {H^{k-1}} \&\&\& q^{2k-3} {H^{k-2}} \\
	\\
	\\
	q^{2k}{\Phi^{k-1}} \&\& q^{2k-1} {H^{k-1}\oplus q^{2k-1} H^{k-1}} \&\&\& q^{2k-3} {H^{k-2}} \\
	\\
	\\
	q^{2k-1} {H^{k-1}} \&\& q^{2k-1} {H^{k-1}} \&\&\& \varnothing
	\arrow["h", from=7-1, to=7-3]
	\arrow[from=7-3, to=7-6]
	\arrow["{\begin{bmatrix}(-1)^kh\\h\end{bmatrix}}", from=7-3, to=4-3]
	\arrow["\iota"', from=7-1, to=4-1]
	\arrow["\epsilon"', from=4-1, to=1-1]
	\arrow["{\begin{bmatrix}\epsilon\tilde{d}^{-k}\\m_L\end{bmatrix}}", from=4-1, to=4-3]
	\arrow["{\begin{bmatrix}h&(-1)^{k-1}h\end{bmatrix}}"', from=4-3, to=1-3]
	\arrow["{\begin{bmatrix}d^{-k}&(-1)^{k-1}d^{-k}\end{bmatrix}}", from=4-3, to=4-6]
	\arrow[from=7-6, to=4-6]
	\arrow["{d^{-k-1}}", from=1-1, to=1-3]
	\arrow["{d^{-k}}", from=1-3, to=1-6]
	\arrow["h", from=4-6, to=1-6]
\end{tikzcd}.\]

\end{lemma}

\begin{proof}
The bottom left square commutes by the bubble and sphere relations and the top left square commutes by the neck cutting relation.
The rest is similar to the preceding lemma.
\end{proof}

\begin{proof}[Proof of \ref{twist-complex}]
The proof follows by induction.
The base case \(k=2\) is given by Lemma \ref{basecase}.

Suppose that the complex \(C^*(A_k)\) is relative homotopic to Proposition \ref{twist-complex}.
Then the complex \(C^*(A_{k+1})\) is relative homotopy to
\begin{equation}
\label{above}
\begin{tikzcd}[ampersand replacement=\&]
	0 \& q^{2k-1}{H^{k-1}} \& q^{2k-3}{H^{k-2}} \& \cdots \& q{H^{0}} \& I \& 0 \\
	0 \& q^{2k} {\Phi^{k-1}} \& q^{2k-2} {\Phi^{k-2}} \& \cdots \& q^2 {\Phi^{0}} \& qH^0 \& 0
	\arrow["{-d^{-1}}", from=1-5, to=1-6]
	\arrow["{}", from=1-6, to=1-7]
	\arrow["{-d^{-k+1}}", from=1-3, to=1-4]
	\arrow["{-d^{-k}}", from=1-2, to=1-3]
	\arrow[from=1-1, to=1-2]
	\arrow["{-d^{-2}}", from=1-4, to=1-5]
	\arrow["{m_L}", from=2-2, to=1-2]
	\arrow["{\tilde{d}^{-k}}"', from=2-2, to=2-3]
	\arrow[from=2-1, to=2-2]
	\arrow["{\tilde{d}^{-k+1}}"', from=2-3, to=2-4]
	\arrow["{m_L}", from=2-3, to=1-3]
	\arrow["{\tilde{d}^{-2}}"', from=2-4, to=2-5]
	\arrow["{}"', from=2-6, to=2-7]
	\arrow["{\tilde{d}^{-1}}"', from=2-5, to=2-6]
	\arrow["{m_L}", from=2-5, to=1-5]
	\arrow["{z}", from=2-6, to=1-6]
\end{tikzcd},
\end{equation}
where \(\tilde{d}^{-1} := m_R\), \(\tilde{d}^{-j} := \phi^{\bullet}_R - \phi^{\bullet}_L\) when \(j\) is even and \(\tilde{d}^{-j} := \phi^{\bullet}_R + \phi^{\bullet}_L - E_1 \phi\), when \(j \neq 1\) is odd.

By the induction hypothesis, the part of \(C^*(A_k)\) excluding the leftmost term is the complex \(q^{-1}C^*(A_{k-1})\).
Thus the portion of \eqref{above} excluding the leftmost column must be isomorphic to \(q^{-1} C^*(A_k)\).
Therefore the entire diagram is isomorphic to
\begin{equation}
\begin{tikzcd}[ampersand replacement=\&]
	0 \& q^{2k-1} {H^{k-1}} \& \& q^{2k-3} {H^{k-2}} \& \cdots \& q{H^{0}} \& I \& 0 \\
	0 \& q^{2k} {\Phi^{k-1}} \& \& q^{2k-1} {H^{k-1}}
	\arrow[from=1-7, to=1-8]
	\arrow["{d^{-1}}", from=1-6, to=1-7]
	\arrow["{d^{-2}}", from=1-5, to=1-6]
	\arrow["{d^{-k+1}}", from=1-4, to=1-5]
	\arrow["{(-1)^{k-1}d^{-k}}", from=1-2, to=1-4]
	\arrow[from=1-1, to=1-2]
	\arrow["{m_L}", from=2-2, to=1-2]
	\arrow["{\epsilon\tilde{d}^{-k}}"', from=2-2, to=2-4]
	\arrow[from=2-1, to=2-2]
	\arrow["{d^{-k}}", from=2-4, to=1-4]
\end{tikzcd}.
\end{equation}

The final square is the middle row of the diagram in Lemma \ref{finalsquare} so in the relative homotopy category, the diagram is isomorphic to
\begin{equation}
	0 \xrightarrow{} q^{2k+1} \uwave{H^{k}} \xrightarrow{d^{-k-1}} q^{2k-1} {H^{k-1}} \xrightarrow{d^{-k}} q^{2k-3}H^{k-2} \xrightarrow{d^{-k+1}}  \cdots \xrightarrow{d^{-2}} q{H^{0}} \xrightarrow{d^{-1}} I \xrightarrow{} 0.
\end{equation}
\end{proof}

\subsection{The complex \(C^*(T_{2,k})\)}
\begin{corollary}
In the relative homotopy category, \(C^*(T_{2,k})\) is isomorphic to
\begin{equation}
	0 \xrightarrow{} q^{2k-1} {\Theta^{k-1}} \xrightarrow{d^{-k}} \cdots \xrightarrow{0} q^{5} {\Theta^2} \xrightarrow{d^{-3}} q^{3} {\Theta^{1}} \xrightarrow{0} q {\Theta^0} \xrightarrow{p} {O \otimes O} \xrightarrow{} 0,
\end{equation}
where \(d^{-1} := p\), \(d^{-j} := 0\) when \(j\) is even, and \(d^{-j} := 2\theta^{\bullet} - E_1 \theta\) when \(j\neq 1\) is odd.
\end{corollary}
\begin{proof}
The web \(I\) is replaced by a pair of oppositely oriented circles denoted by \(O \otimes O\) and each \(H\) is replaced by theta webs
\[\Theta := \includegraphics[valign=c]{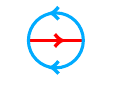},
\Theta^n := \includegraphics[valign=c]{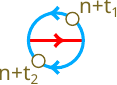}.\]
The foam \(h\) is replaced with the theta foam and the foams \(h^{\bullet}_L\) and \(h^{\bullet}_R\) are replaced by dotted theta foams
\[\theta := \includegraphics[valign=c]{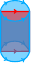}, \quad
\theta^{\bullet} := \includegraphics[valign=c]{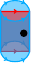},\]
where the dot in \(\theta^\bullet\) is on the forward most facet.
The zip foam \(z\) is replaced by a singular pair of pants
\[p := \includegraphics[valign=c]{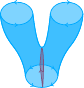}.\]
Finally, the differentials collapse so that \(d^{-1} = p\), \(d^{-j} = 0\) when \(j\) is even, and \(d^{-j} = 2\theta^{\bullet} - E_1 \theta\) when \(j\neq 1\) is odd.
\end{proof}

\section{The \(\slt\)-structure of the homology \(H^*(T_{2,k})\)}
\label{sl2-computation}

\subsection{The homology of \(T_{2,k}\) in terms of polynomials}
In this section, we compute the homology of the complex \(C^*(T_{2,k})\) as a \(\mathcal{U}(\slt)\#R_2\)-module and decompose it into a direct sum of \(\slt\) indecomposables.
As a first step, it is useful to identify the homology of the unknot with a submodule of the ring of algebraic functions equipped with the action of \(\slt\) extended via the differential operators in \ref{action-poly}.

The state spaces for \(O\) and \(\Theta\) are generated by the dotted and undotted cup and theta cup foams respectively:
\[\theta_\iota := \includegraphics[valign=c]{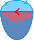} \quad
\theta^{\bullet}_\iota := \includegraphics[valign=c]{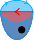}\quad
o_\iota := \includegraphics[valign=c]{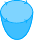}\quad
o^{\bullet}_\iota := \includegraphics[valign=c]{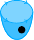}.\]
The dot in \(\theta^\bullet_\iota\) belongs to the forward most facet.

\begin{lemma}
\label{poly-rep}
The map 
\(\phi: H^*(O) \to (x-y)^{-1/2}\mathbf{k}[x,y]\) defined by
\begin{equation}
\begin{gathered}
o_\iota \mapsto (x-y)^{-1/2}\\
o^\bullet_\iota \mapsto (x-y)^{-1/2}x\\
\end{gathered}
\end{equation}
is a graded \(\mathcal{U}(\slt)\# R_2\)-module isomorphism.
\end{lemma}
\begin{proof}
The module \((x-y)^{-1/2}\mathbf{k}[x,y]\) is isomorphic to  \(\mathbf{k}[x,y]\) and the latter is
a free \(R_2\)-module with generators \(\{1,x\}\).
Thus \(\phi\) is an isomorphism of \(R_2\)-modules.

The map \(\phi\) additionally intertwines the action of \(\mathbf{k}[x]\) on \(O\) given by adding decorations with the action of \(\mathbf{k}[x]\) on \(\mathbf{k}[x,y]\) given by multiplication.
Moreover, the action of \(\slt\) on decorations in \(\mathbf{k}[x]\) agrees with the restriction of the \(\slt\) action along \(\mathbf{k}[x] \to \mathbf{k}[x,y]\).
Thus it is enough to show that \(\phi(g \cdot o_\iota) = g \cdot \phi(o_\iota)\):
\begin{equation}
\begin{aligned}
\mathbf{e} \cdot o_\iota &= 0 &
\mathbf{e} \cdot (x-y)^{-1/2} &= 0 \\
\mathbf{h} \cdot o_\iota &= o_\iota &
\mathbf{h} \cdot (x-y)^{-1/2} &= (x-y)^{-1/2} \\
\mathbf{f} \cdot o_\iota &= -\frac{1}{2} E_1 o_\iota &
\mathbf{f} \cdot (x-y)^{-1/2} &= -\frac{1}{2} (x+y)(x-y)^{-1/2}.
\end{aligned}
\end{equation}
\end{proof}

\begin{lemma}
\label{poly-rep-theta}
The map 
\(\psi: H^*(q\Theta^0) \to (x-y)^{-1}\mathbf{k}[x,y]\) defined by
\begin{equation}
\begin{gathered}
\theta_\iota \mapsto (x-y)^{-1}\\
\theta^\bullet_\iota \mapsto (x-y)^{-1}x\\
\end{gathered}
\end{equation}
is a graded \(\mathcal{U}(\slt)\# R_2\)-module isomorphism.
\end{lemma}

\begin{proof}
The proof is identical to the proof of Lemma \ref{poly-rep} with minor adjustments for the action on the theta cup.
To compute the \(\slt\) action on \(\theta_\iota\), decompose it as a pair of cups and a zip.
The crucial point is that the twists in \(\Theta^0\) cancel the action on the zip so that the resulting action does not depend on \(t_1\) or \(t_2\).
\end{proof}

\begin{lemma}
\label{poly-rep-twisted}
Let \(O^n\) denote the unknot with a hollow twist and a solid twist both labelled by \(n \in \mathbf{k}\).
The maps
\(\psi_n: H^*(q^{2n+1}\Theta^n) \to (x-y)^{n-1}\mathbf{k}[x,y]\) defined by
\begin{equation}
\begin{gathered}
\theta_\iota \mapsto (x-y)^{n-1}\\
\theta^\bullet_\iota \mapsto (x-y)^{n-1}x,\\
\end{gathered}
\end{equation}
and
\(\phi_n: H^*(q^{2n}O^n) \to (x-y)^{n-1/2}\mathbf{k}[x,y]\) defined by
\begin{equation}
\begin{gathered}
o_\iota \mapsto (x-y)^{n-1/2}\\
o^\bullet_\iota \mapsto (x-y)^{n-1/2}x\\
\end{gathered}
\end{equation}
are graded \(\mathcal{U}(\slt)\# R_2\)-module isomorphisms.
\end{lemma}

\begin{proof}
For any \(g \in \slt\) and \(n, m\in \mathbf{k}\),
\begin{equation}
\begin{aligned}
g \cdot (x-y)^{n-m} &=
\left(g\cdot (x-y)^{-m}\right)(x-y)^n + \left(g \cdot (x-y)^n \right)(x-y)^{-m} \\
&= \left(g\cdot (x-y)^{-m}\right)(x-y)^n + \left(\frac{g \cdot (x-y)^n}{(x-y)^n} \right)(x-y)^{n-m}.
\end{aligned}
\end{equation}
Thus the action of \(\slt\) on \((x-y)^{n-m}\mathbf{k}[x,y]\) can be viewed as a twist of the action on \((x-y)^{-m}\mathbf{k}[x,y]\) by the multiplication by \(\frac{g\cdot (x-y)^n}{(x-y)^n}\) endomorphism.
Lemma \ref{poly-rep-theta} identifies the extra twists on \(O^n\) and \(\Theta^n\) with multiplication by the same polynomials:
\begin{equation}
\begin{aligned}
\mathbf{e} &\mapsto 0 = \frac{\mathbf{e} (x-y)^n}{(x-y)^n}\\
\mathbf{h} &\mapsto -2n = \frac{\mathbf{h} (x-y)^n}{(x-y)^n}\\
\mathbf{f} &\mapsto n(x+y) =\frac{\mathbf{f} (x-y)^n}{(x-y)^n}.
\end{aligned}
\end{equation}
\end{proof}

\begin{theorem}
\label{homaspoly}
As a graded \(\mathcal{U}(\mathfrak{sl}_2)\# R_2\)-module 
\begin{align*}
H^0(T_{2,k})     &\cong (x-y)^{-1}\mathbf{k}[x,y], \\
H^{-j}(T_{2,k}) &\cong 0 \text{ for all odd } j, \\
H^{-j}(T_{2,k}) &\cong (x-y)^{j-2}\mathbf{k}[x,y]/(x-y)^{j-1}\mathbf{k}[x,y] \text{ for all even } 0 < j < k,\\
H^{-k}(T_{2,k})   &\cong (x-y)^{k-2}\mathbf{k}[x,y] \text{ if } k \text{ is even}.
\end{align*}
\end{theorem}

\begin{remark}
The definition of the \(\slt\)-action depended on parameters \(t_1, t_2 \in \mathbf{k}\).
The homology \(H^*(T_{2,k})\), however, is independent of them.
It is unclear if the homology depends on the parameters in general.
\end{remark}

\begin{proof}
Since all of the even differentials are zero, \(C^*(T_{2,k})\) is a direct sum of two-term complexes
\begin{equation}
\begin{aligned}
q\Theta^0 &\xrightarrow{p} O \otimes O \\
q^{2n+3}\Theta^{n+1} &\xrightarrow{2\theta^{\bullet} -E_1 \theta} q^{2n+1}\Theta^{n}.
\end{aligned}
\end{equation}
The first complex is found in homological degrees \(-1\) and \(0\).
The second complex is found in homological degrees \(-n-2\) and \(-n-1\).
If \(k\) is even there is an additional direct summand, \(q^{2k-1}\Theta^{k-1}\), in homological degree \(-k\).

Observe that the first complex is the same as the complex associated to an unknot with a single positive twist.
By the Reidemeister I move, it is relative homotopy equivalent to \(qO^{-1/2}\) in homological degree 0.

By Lemma \ref{poly-rep-twisted}, the second complex is isomorphic to
\begin{equation}
(x-y)^{n}\mathbf{k}[x,y] \to (x-y)^{n-1}\mathbf{k}[x,y]
\end{equation}
where the map is identified with the inclusion.
The homology of the complex is
\((x-y)^{n-1}\mathbf{k}[x,y]/(x-y)^n\mathbf{k}[x,y]\)
concentrated in degree \(-n-1\).
\end{proof}


\subsection{Decomposition into Indecomposables}

\subsubsection{Review of Integral Highest Weight Theory for \(\slt\)}
Let \(V\) be a representation of \(\slt\).
A vector \(v \in V\) is called \textit{highest weight} if \(\mathbf{e} v = 0\).
If the vector subspace \(V_\lambda := \{ v \in V \mid \mathbf{h} v = \lambda v\}\) is non-zero, it is called the \(\lambda\)-\textit{weight space} of \(V\), for some \(\lambda \in \Z\).
If \(V\) is a direct sum of its weight spaces, then \(V\) is called a \textit{weight module}.

The finitely generated, indecomposable highest weight \(\slt\)-modules with finite dimensional weight spaces  and a fixed highest weight \(\lambda \in \Z\) are classified into the following isomorphism types:
\begin{itemize}
\item The finite dimensional simple \(L(\lambda)\) of dimension \(\lambda+1\) (if \(\lambda \geq 0\)).
\item The infinite-dimensional Verma module \(M(\lambda)\).
If \(\lambda < 0\), then \(M(\lambda)\) is simple.
Otherwise it contains a simple submodule isomorphic to \(M(-\lambda-2)\) with a simple quotient isomorphic to \(L(\lambda)\).
\item The dual Verma module \(M^*(\lambda)\).
If \(\lambda < 0\), then \(M^*(\lambda) \cong M(\lambda)\) is simple.
Otherwise it contains a finite dimensional submodule \(L(\lambda)\) with a simple quotient isomorphic to \(M(-\lambda-2)\).
\item The projective cover \(P(-\lambda-2)\) of the Verma \(M(-\lambda-2)\).
The projective cover has a submodule isomorphic to \(M(\lambda)\) with quotient isomorphic to \(M(-\lambda-2)\) and a submodule isomorphic to \(M(-\lambda-2)\) with quotient isomorphic to \(M^*(\lambda)\).
\end{itemize}

The only non-split extensions between indecomposable \(\slt\)-weight modules are the ones described as quotients above. For more details see \cite{Mazorchuk_sl2, Humphreys_O}.

\begin{theorem}
\label{irreps}
The homology \(H^*(T_{2,k})\) has the following direct sum decomposition into \(\mathfrak{sl}_2\)-indecomposables:
\begin{align*}
H^0(T_{2,k})     &\cong \bigoplus_{\lambda=-1}^{+\infty} M^*(-2\lambda) \\
H^{-j}(T_{2,k}) &\cong 0 \text{ for all odd } j \\
H^{-j}(T_{2,k}) &\cong M^*(-2(j-2)) \text{ for all even } 0 < j < k\\
H^{-k}(T_{2,k})   &\cong \bigoplus_{\lambda=k-2}^{+\infty}  M^*(-2\lambda) \text{ if } k \text{ is even. }
\end{align*}
\end{theorem}

\begin{proof}
Consider the \(\slt\)-module \(A := \mathbf{k}[x,y, (x-y)^{-1}]\) with a filtration
given by the submodules \(A^\lambda := (x-y)^\lambda\mathbf{k}[x,y]\).
The theorem will follow from a decomposition of \(A\) since the submodules \(A^\lambda\) and the quotients \(A^{\lambda}/A^{\lambda+1}\) are exactly the homology modules found in Theorem \ref{homaspoly}.
The argument is outlined here and details are provided in Lemma \ref{verma-cases} and Proposition \ref{direct-sum}.

Recall that the \(\lambda\)-weight space of \(\mathbf{k}[x,y]\) is the subspace of homogeneous degree \(\lambda\) polynomials.
This implies that \(\dim(A^{\lambda}_{\nu+1}) = \dim(A^{\lambda}_{\nu}) + 1\) for any weight \(\nu\) and thus
the associated graded pieces \(A^{\lambda}/A^{\lambda+1}\) are one dimensional in each weight.
Each quotient module \(A^{\lambda}/A^{\lambda+1}\) is isomorphic to a dual Verma \(M^*(-2\lambda)\) with highest weight vector \((x-y)^\lambda\).
In fact, the submodule generated by \((x-y)^{\lambda}\) is simple in the full module \(A\), so there are no further extensions.
Together this implies that \(A\) is the direct sum of dual Vermas and the decomposition of the homology is immediate.
\end{proof}

\begin{lemma}
\label{verma-cases}
For all \(\lambda \in \Z\), the vector \((x-y)^\lambda \in \mathbf{k}[x,y,(x-y)^{-1}]\) is a highest weight vector and the submodule it generates is isomorphic to a simple module.
\end{lemma}

\begin{proof}
To show that \((x-y)^\lambda\) is highest weight it is sufficient to show that \((x-y)\) is. Indeed,
\(
\mathbf{e} (x-y) = -1 - (-1) = 0
\).

Let \(V(-2\lambda)\) be the submodule generated by \((x-y)^\lambda\).
If \(\lambda > 0\) then \(V(-2\lambda)\) has negative highest weight \(-2\lambda\) so it must be isomorphic to the simple module \(M(-2\lambda)\).
To show that \((x-y)^{\lambda}\) generates \(L(-2\lambda)\) when \(\lambda \leq 0\), it is enough to show that
\(\mathbf{f}^{-2\lambda+1}(x-y)^{\lambda} = 0\).
The proof proceeds by induction.

In the base case, \((x-y)^0 = 1\) generates the trivial module.
Assume that \({(x-y)^{\lambda}}\) generates a finite dimensional simple \(L(-2\lambda)\) so that
\(\mathbf{f}^{-2\lambda}(x-y)^{\lambda} \neq 0\)
and
\({\mathbf{f}^{-2\lambda+1}(x-y)^{\lambda}} = 0\).
Now apply the generalized Leibniz rule:
\begin{equation}
\begin{aligned}
\label{lieb}
\mathbf{f}^{-2\lambda+3}(x-y)^{\lambda-1} &= 
\mathbf{f}^{-2\lambda+3}((x-y)^{-1}(x-y)^{\lambda}) \\ &=
\sum_{k=0}^{-2\lambda+3} {-2\lambda + 3 \choose k} \mathbf{f}^k(x-y)^{-1} \mathbf{f}^{-2\lambda+3-k} (x-y)^{\lambda}.
\end{aligned}
\end{equation}
By the induction hypothesis, \(\mathbf{f}^{-2\lambda+3-k}(x-y)^{\lambda} = 0\) whenever \(k \leq 2\) and
direct computation shows that
\(\mathbf{f}^k(x-y)^{-1} = 0\) when \(k \geq 3\).
Therefore all terms in the sum vanish.
\end{proof}

\begin{proposition}
\label{direct-sum}
As an \(\slt\)-representation,
\[
\mathbf{k}[x,y, (x-y)^{-1}] \cong \bigoplus_{\lambda \in \Z} M^*(-2\lambda).
\]
\end{proposition}
\begin{proof}
When \(\lambda > 0\), \(L(-2\lambda) \cong M^*(-2\lambda)\).
When \(\lambda \leq 0\), the submodule \(L(-2\lambda)\) is finite dimensional and we will show that it is contained in a submodule isomorphic to \(M^*(-2\lambda)\) by finding a vector \(w\) with weight \(2\lambda-2\) such that \(\mathbf{e} w = \mathbf{f}^{-2\lambda} (x-y)^\lambda\).
Comparing the dimensions of the weight spaces and number of dual Vermas shows that there are no further extensions.

Recall the operators \(L_k\) from Definition \ref{action-poly}.
A component of \(L_{-2\lambda+1}(x-y)^\lambda\) will satisfy the condition required of \(w\).
First note that if \(k \geq -1\), \(L_k\) acts on the submodule \((x-y)^n\mathbf{k}[x,y]\):
For any \(p(x,y) \in \mathbf{k}[x,y]\),
\begin{equation*}
\begin{aligned}
   L_k \cdot (x-y)^n p(x,y)
&= (x-y)^n L_k \cdot p(x,y) + n (x-y)^{n-1} (x^{k+1}-y^{k+1})p(x,y) \\
&= (x-y)^n L_k \cdot p(x,y) + n (x-y)^{n} h_k(x,y) p(x,y) \\
&=(x-y)^{n}\left[L_k \cdot p(x,y) + n h_k(x,y)p(x,y) \right].
\end{aligned}
\end{equation*}

Suppose \(v\) is a highest weight vector of weight \(\lambda\).
The operators \(\mathbf{e}\) and \(L_n\) satisfy the relation
\begin{equation}
\label{witt-rel}
\mathbf{e} L_n v = [\mathbf{e} ,L_n]v + L_n\mathbf{e}v = -(n+1)L_{n-1} v.
\end{equation}
The submodule \(A^\lambda\) is preserved by \(\mathbf{e}\)
so inductively applying (\ref{witt-rel}) implies that \(L_n v\) has a component in \(A^{\lambda}/A^{\lambda+1}\) for all \(\lambda\).
On the other hand,
the \(\mathfrak{sl}_2\)-submodule generated by \(v\) is simple so it lies entirely in the quotient \(A^\lambda/A^{\lambda+1}\).
Moreover \(A^{\lambda}/A^{\lambda+1}\) is one-dimensional,
so the component of \(L_{-2\lambda}v\) in \(A^{\lambda}/A^{\lambda+1}\) is a scalar multiple of \(\mathbf{f}^{-2\lambda}v\).
Finally, this implies that the component \(w\) of \(L_{-2\lambda+1}v\) in \(A^{\lambda}/A^{\lambda+1}\) or at least a scalar multiple of it, satisfies \(\mathbf{e} w =  \mathbf{f}^{-2\lambda}v\).
\end{proof}


In Lemma \ref{poly-rep-twisted}, the representation \((x-y)^{n-1/2}\mathbf{k}[x,y]\) was also found among the homologies of simple knots.
Similar results apply to modules of this form.

\begin{proposition}
\label{direct-sum-half}
For any \(\lambda \in \frac{1}{2}+\Z\), the vector \((x-y)^\lambda\)
is a highest weight vector in \((x-y)^{-1/2}\mathbf{k}[x,y,(x-y)^{-1}]\) and
the submodule it generates is isomorphic to the Verma module \(M(-2\lambda)\).
If \(\lambda \leq -\frac{1}{2}\), then the Verma submodule generated by \((x-y)^{-\lambda+1}\) is the unique Verma submodule of the one generated by \((x-y)^{\lambda}\).
As an \(\slt\)-representation,
\[
(x-y)^{-1/2}\mathbf{k}[x,y, (x-y)^{-1}] \cong \bigoplus_{\lambda\leq 0} P(2\lambda-1).
\]
\end{proposition}

\begin{proof}
An induction argument similar to the one used in Lemma \ref{verma-cases} shows that \(\mathbf{f}^{-2\lambda+1}(x-y)^{\lambda}\)
is a non-zero scalar multiple of \((x-y)^{-\lambda+1}\).
This is enough to prove the first two statements.
When \(\lambda \geq -\frac{1}{2}\), the argument from Proposition \ref{direct-sum} applies with minor modification and shows that the copy of \(M(-2\lambda)\) generated by \((x-y)^{\lambda}\) is in extension with a submodule isomorphic to \(M(2\lambda-1)\).
The unique extension between these modules is \(P(2\lambda-2)\).
The vector \((x-y)^{1/2}\) generates the Verma \(M(-1)\) which is projective and is not in any non-trivial extensions, so \(P(-1)\) appears in the direct sum as well.
Details are left to the reader.
\end{proof}

\begin{corollary}
As an \(\slt\)-representation,
\[
H^*(O) \cong (x-y)^{-1/2}\mathbf{k}[x,y] \cong P(-3) \oplus P(-1) \oplus \bigoplus_{\lambda \leq -2} M^*(2\lambda-1).
\]
\end{corollary}

\begin{remark}
The proof of Proposition \ref{direct-sum} depends entirely on the filtration given by \((x-y)\)-divisibility.
Propositions \ref{direct-sum} and \ref{direct-sum-half} show that the associated graded components are dual Verma modules.
This filtration is essential to topological applications where \((x-y)\), often called \(G\) or \(H\), is used to define numerical invariants of links \cite{sano, Iltgen_2025}.
The presence of an \(\slt\)-action whose structure is controlled by \((x-y)\) may produce finer invariants with applications to topology.
\end{remark}

\bibliographystyle{alphaurl}
\bibliography{main}

\end{document}